\documentclass{amsart}
\usepackage{amsfonts,amssymb,amsmath, dsfont}
\usepackage{upgreek}
\newtheorem{theorem}{Theorem}[section]
\newtheorem{proposition}[theorem]{Proposition}
\newtheorem{lemma}[theorem]{Lemma}
\newtheorem{corollary}[theorem]{Corollary}
\theoremstyle{remark}

\def \R {{\mathbb R}}
\def \Z {{\mathbb Z}}

\def \ba {\begin{array}}
\def \ea {\end{array}}

\def \cF {{\mathcal F}}

\def\one{\rlap{\mbox{\small\rm 1}}\kern.15em 1}

\def \un {\underline}
\def \ov {\overline}

\def \Ll {\left}
\def \Rr {\right}
\def \subset {\subseteq}

\def \emptyset {\varnothing}

\begin{document}

\title{Random walks generated by equilibrium contact processes}

\author[T.\ Mountford and  M.E.\ Vares]{Thomas Mountford\textsuperscript{1} and Maria E.\ Vares\textsuperscript{2}
}
\footnotetext[1]{\'Ecole Polytechnique F\'ed\'erale de Lausanne,
D\'epartement de Math\'ematiques,
1015 Lausanne, Switzerland}
\footnotetext[2]{Instituto de Matem\'atica, Universidade Federal do Rio de Janeiro, Av. Athos da Silveira Ramos 149, CT- Bloco C, Cidade Universit\'aria, 21941-909 Rio de Janeiro, RJ, Brazil}

\begin{abstract}
We consider dynamic random walks where the nearest neighbour jump rates are determined by an underlying supercritical contact process in equilibrium.
This has previously been studied by \cite{dHS} and \cite{dHSS}.  We show the CLT for such a random walk, valid for all supercritical infection rates for the
contact process environment.
\bigskip


\medskip


\end{abstract}

\maketitle
%
%
%
%
%
%
%

\section{Introduction}

\noindent In this note we consider a random motion ${(X_{t})}_{t \geq 0}$ in ${\Z}$ generated by a supercritical one dimensional
contact process ${(\xi_{t})}_{t \geq 0}$ in upper equilibrium $\bar \nu$. We suppose that the motion ${(X_{t})}_{t \geq 0}$ performs
nearest neighbour jumps with rate depending on the local values of $\xi_{t}$: there exist $r_0 < \infty$ and functions
$g_{1}$ and $g_{-1}$ that depend only on the spins within $r_0$ of the origin so that for all $t$, $i=\pm 1$,
$$
P(X_{t+h}= X_{t} + i \vert X_{s}, \
\xi_s, s \leq t ) = h g_{i} ( \theta_{X_t}  o \xi_{t}) - o(h)
$$
as $h \to 0$, where $(\theta_y o \xi)(x)=\xi(y+x)$ for all $x,y$.  By contrast, the evolution of process $X$ does not affect that of the contact process $\xi$.\\

{\it Remark:} For simplicity we take $X$ to be a nearest neighbour random walk and also $\xi $
to be a nearest neighbour symmetric contact process.  The approach and result given here
extend without difficulty to random walks whose jumps are finite range (with all jump rates being
appropriate shifts of cylinder functions of $\xi$).  Equally with a bit more care the arguments can be adapted to deal with finite range contact processes.

\noindent Our result is that

\begin{theorem} \label{thm1}
For all $\lambda > \lambda_{c}$ and any non trivial (i.e. non identically zero)  $g_{i}$ as above, there exist $\mu\in \R$ and
$\alpha>0$ so that, as $t \to +\infty$,

$$\frac{X_{t} - \mu t}{\alpha \sqrt{t} } \stackrel{D}{\to} N(0,1).$$
\end{theorem}

\noindent This result has already been shown for $\lambda$ large in the case of $r_0 =0$, see \cite{dHS} and \cite{dHSS}, by a nice regeneration argument. We exploit in this article the strong regeneration properties of ${(\xi_{t})}_{t \geq 0}$ but in a different way, though we also embed an i.i.d. sequence of r.v.s in our process.  To our knowledge the first central limit theorem for the contact process was due to \cite{GP} who considered the position of the rightmost occupied site for a one sided supercritical contact process.  A beautiful alternative proof was produced by \cite{K} (who wrote his approach explicitly for oriented percolation).  The central limit proof of \cite{dHS} is in this tradition.

\noindent We suppose that the process $\xi$ is generated by a Harris system as is usual. Details will be supplied in the next section.

\noindent The process $X$ is generated by a Poisson process $N^{X}$ of rate $M^\prime  > {\parallel g_{1}\parallel}_{\infty}+ {\parallel g_{-1} \parallel}_{\infty}$ and associated  i.i.d. uniform [0,1] r.v.s ${\{U_{i}\}}_{i \geq 1}$:  if $t \in N^{X}$ is the i'th Poisson point, then a jump from $X_{t-}$ to $X_{t-} -1$ occurs only if $U_{i} \in [0, g_{-1} (\theta_{X_{t-}} o\xi_{t-})/M^\prime]$ and jumps to $X_{t-} +1$ only if $U_{i} \in [1-g_{1} (\theta_{X_{t-}} o\xi_{t-} )/M^\prime, 1].$

\noindent Thus, irrespective of the behaviour of ${(\xi_{t})}_{t \geq 0}$ over a time interval $I$, if $N^{X} \cap I = \emptyset$ then $X$ makes no jumps over time interval $I$.
We now fix throughout the paper $M > M ^\prime$.

In the following we will use the expression dynamic random walk to denote a pair $(\xi,X)$ which evolve according to these stipulated rules.
We will say that a pair $(\xi,X)$ is a piecewise dynamic random walk if it evolves according to the given rules on time intervals
$[\beta_i,\beta_{i+1})$ with $\beta_i$ increasing to infinity, but may have a global jump in the pair $(\xi,X)$ at times $\beta_i$.
This will be clear in Section \ref{sec-5}.

\section{A reminder on the contact process}
\label{s:remind}
\setcounter{equation}{0}
The contact process with parameter $\lambda > 0$ on a connected graph $G = (V, E)$ is a continuous-time Markov process $(\xi_t)_{t \geq 0}$ with state space $\{0,1\}^V$ and generator
\begin{equation}  \label{eq1gen}
\Omega f(\xi) = \sum_{x \in V} \left(f(\phi_x\xi) - f(\xi) \right) + \lambda \cdot \sum_{e\in E}  \left(f(\phi_e\xi) - f(\xi) \right),
\end{equation}
where $f$ is any local function on $\{0,1\}^V$ and, given $x \in V$ and $\{y, z\} \in E,$ we define $\phi_x\xi,\; \phi_{\{y,z\}}\xi \in \{0,1\}^V$ by
$$\phi_x\xi(w) = \left|\begin{array}{ll}0 &\text{if } w = x;\\\xi(w)&\text{otherwise;} \end{array}\right.\qquad \phi_{\{y,z\}}\xi(w) = \left|\begin{array}{ll}\max(\xi(y),\xi(z))&\text{if } w \in \{y,z\};\\\xi(w) &\text{otherwise.} \end{array} \right.$$
Given $A \subset V$, we write $(\xi^A_t)_{t \ge 0}$ to denote the contact process started from the initial configuration that is equal to 1 at vertices of $A$ and 0 at other vertices. When we write $(\xi_t)$, with no superscript, the initial configuration will either be clear from the context or unimportant. We often abuse notation and identify configurations $\xi \in \{0,1\}^V$ with the corresponding sets $\{x\in V: \xi(x) = 1\}$.

The contact process is a model for the spread of an infection in a population. Vertices of the graph (sometimes referred to as \textit{sites}) represent individuals. In a configuration $\xi \in \{0,1\}^V$, individuals in state 1 are said to be \textit{infected}, and individuals in state 0 are \textit{healthy}. Pairs of individuals that are connected by edges in the graph are in proximity to each other in the population. The generator (\ref{eq1gen}) gives two types of transition for the dynamics. First, infected individuals \textit{heal} with rate 1. Second, given two individuals in proximity so that one is infected and the other is not, with rate $\lambda$ there occurs a \textit{transmission}, as a consequence of which both individuals end up infected.

The configuration $\underline{0} \in\{0,1\}^V$ that is equal to zero at all vertices is a trap for $(\xi_t)$. For certain choices of the underlying graph $G$ and the parameter $\lambda$, it may be the case that the probability of the event $\{\underline{0} \text{ is never reached}\}$ is positive even if the process starts from finitely many infected sites. In fact, whether or not this probability is positive does not depend on the set of initially infected sites, as long as this set is nonempty and finite. We say that the process \textit{survives} if this probability is positive; otherwise we say that the process \textit{dies out}.  Survival or not depends on the value of the parameter $\lambda $.  As is intuitive, there is a value $\lambda_c$ (depending on $G$) so that there is survival above $\lambda_c$  and nonsurvival below.

We now recall the graphical construction of the contact process and its self-duality property. Fix a graph $G = (V,E)$ and $\lambda > 0$. We take the following family of independent Poisson point processes on $(- \infty,\infty)$:
$$\begin{array}{ll}
(D^x): x \in V &\text{with rate } 1;\\
(N^e): e \in E &\text{with rate } \lambda.\end{array}$$
Let $H$ denote a realization of all these processes. Given $x,y\in V,\; s \leq t$, we say that $x$ and $y$ are connected by an \textit{infection path in $H$} (and write $(x,s)\leftrightarrow (y,t)$ in $H$) if there exist times $t_0 = s < t_1 < \cdots < t_k = t$ and vertices $x_0 = x, x_1, \ldots, x_{k-1} = y$ such that
\begin{itemize}
\item[$\bullet$] $D^{x_i} \cap (t_i,\; t_{i+1}) = \emptyset$ for $i = 0, \ldots, k - 1$;
\item[$\bullet$] $\{x_i,x_{i+1}\}\in E$ for $i = 0, \ldots, k-2$;
\item[$\bullet$] $t_i \in N^{x_{i-1}, x_i}$ for $i = 1, \ldots, k-1$.
\end{itemize}
Such a collection will be called a {\it path} from $(x,s)$ to $(y,t)$  (here and elsewhere, we drop the dependence on $H$ if a Harris system is given).
Points of the processes $(D^x)$ are called \textit{death marks} and points of $(N^e)$ are \textit{links}; infection paths are thus paths that traverse links and do not touch death marks. $H$ is called a \textit{Harris system}; we often omit dependence on $H$. For $A, B \subset V$, we write $A\times\{s\} \leftrightarrow B \times \{t\}$ if $(x,s)\leftrightarrow (y,t)$ for some $x \in A$, $y \in B$. We analogously write $A\times \{s\} \leftrightarrow (y,t)$ and $(x,s) \leftrightarrow B\times\{t\}$. Finally, given set $C \subset V \times (- \infty,\infty)$, we write $A \times \{s\} \leftrightarrow B \times \{t\}$ \textit{inside} $C$ if there is an infection path from a point in $A\times \{s\}$ to a point in $B\times\{t\}$ which is entirely contained in $C$.

Given $A \subset V$, put
\begin{equation}\label{eq1harris} \xi^A_t(x) = \mathds{1}_{\{A \times \{0\} \leftrightarrow (x,t)\}} \text{ for } x \in V,\; t \geq 0\end{equation}
(here and in the rest of the paper, $\mathds{1}$ denotes the indicator function). It is well-known that the process $(\xi^A_t)_{t\geq 0} = (\xi^A_t(H))_{t\geq 0}$ thus obtained has the same distribution as that defined by the infinitesimal generator (\ref{eq1gen}).  The advantage of (\ref{eq1harris}) is that it allows us to construct in the same probability space versions of the contact processes with all possible initial distributions. From this joint construction, we also obtain the \textit{attractiveness} property of the contact process: if $A \subset B \subset V$, then $\xi^A_t(H) \subset \xi^B_t(H)$ for all $t$. From now on, we always assume that the contact process is constructed from a Harris system.
In discussing dynamic random walks, it will be understood that the Poisson process $N^X$
and associated uniform random variables also are part of the Harris system.

Now fix $A \subset V,\; t  \in \R$ and a Harris system $H$. Let us define the \textit{dual process} $(\hat \xi^{A, t}_s)_{0 \leq s <  \infty}$ by
$$\hat \xi^{A,t}_s(y) = \mathds{1}_{\{(y,t-s)\leftrightarrow A \times \{t\} \text{ in } H\}}.$$
If $A = \{x\}$, we write $(\hat \xi^{x,t}_s)$.
This process satisfies two important properties. First, the distribution of $(\hat \xi^{x,t}_{s}$ $s  \geq 0$) is the same as that of a contact process with same initial configuration. Second, it satisfies the \textit{duality equation}
\begin{equation}\xi^A_t \cap B \neq \emptyset \text{ if and only if } A \cap \hat \xi^{B,t}_t \neq \emptyset. \end{equation}
In particular,
\begin{equation}\xi^{\underline 1}_t(x) = 1 \text{ if and only if } \hat \xi^{x,t}_t \neq \emptyset, \end{equation}
where $(\xi^{\underline 1}_t)$ is the process started from full occupancy.

Also (for $\lambda>  \lambda_c $) if we put $\xi_0(x) \ = \ 1$ if and only if $\hat \xi^{x, 0}_.$ never dies out, then configuration $\xi_0 $ has the upper equilibrium distribution.

\medskip
We will talk of a contact process $\{\xi_t\}_{t\geq 0}$ restricted to $ R \subset V \times \R $ to mean the contact process generated by Harris system paths that are entirely contained in $R$.
This is interpreted to signify that $\xi_t(x) = 0$ for each $(x,t) \notin R$.   We remark that if $R_1 $ and $R_2 $ are disjoint, then conditional upon initial configurations, two contact processes
restricted respectively to $R_1 $ and $R_2$ are independent.   When necessary we use the notation
$\xi^{R}$ or $\xi^{A,R}$ to denote contact processes restricted to space time regions $R$, with $A$ standing for the initial configuration.

We use the suffix $t$ to denote contact processes run from time $t$.

From now on we will consider the supercritical contact process ($\lambda >\lambda_c$)  on
the integer lattice, $V = \Z$ with $E$ the set of nearest neighbour edges.

We now recall classical results about the contact process on the line. The proposition
below can be found in \cite{DG} or Theorem 3.23, chapter VI of \cite{lig85}. 
\begin{proposition}
\label{cpinterval}
There exists a constant $c_1 \in (0, \infty ) $ so that for $\tau $ the stopping
time equal to the first hitting time of $\underline 0 $
for the process, we have
$$
(i) \qquad \quad \! P^{\xi_0}( \tau < \infty ) \ < \ \frac{1}{c_1} e^{-c_1 \sum_x \xi_0(x)}
$$
and
$$
(ii) \qquad \quad P^{\xi_0}( t < \tau < \infty ) \ < \ \frac{1}{c_1} e^{-c_1 t},
$$
for any configuration $\xi_0$.
\end{proposition} 
One important consequence of (ii) above (indeed of the slightly weaker version 
when $\xi_0$ has only one occupied site) is the fact that if instead of considering as 
occupied (at a given time) sites whose dual survives forever, we consider sites whose dual survives to large time $t$, 
then the resulting configuration has a distribution very close to equilibrium. 

We have that a contact process $ (\xi^{x}_t)_{t \geq 0} $ has exponentially small (in $t$) chance of
surviving until time $t$ but subsequently dying out.  Furthermore
by considering large deviations of the rightmost descendant $r^{x}_t $ and leftmost
descendant $l^{x}_t $ (see \cite{DSch} and Theorem 3.23, of \cite{lig85}), we have that
\begin{lemma}
\label{lemreg}
There exists $h_1 > 0 $ so that
$$
P(H(t))  < \frac{1}{h_1} e^{-h_1t}
$$
for the event $H(t)$: $\xi^{x}_t \ne \emptyset $ but either
    $ | \xi^{x}_t \cap (x,x+t ) | < h_1t$ or
    $ | \xi^{x}_t \cap (x-t,x ) | < h_1t,$
where  $|\cdot|$ refers to the cardinality.
\end{lemma}
{\it Remark:} The second statement in Proposition \ref{cpinterval} follows at once from the first and the above lemma.

The classical renormalization argument that compares the contact process with supercritical oriented percolation,
see for instance the proof of \cite[Corollary~VI.3.22]{lig85} and classical contour arguments
for oriented percolation (see e.g. \cite{Du1}) give the following

\begin{lemma}
\label{hit}
Given $\gamma $ and $ \theta $ both strictly greater than $\beta$, there exists constant $c_2$
so that for $( \xi_t: t \geq 0)$ the contact process restricted to rectangle $[0,L] \times [0,T]$, one has: if\\
(i) $\xi_0 $ has no gaps of size $L^ \beta $\\
 and
(ii) the dual $\hat {\xi}^{x,T} _{T- L^\gamma} $ has cardinality at least $L^{\theta}$,\\
then the conditional probability that $\xi_T(x) = 1 $ is at least
$$
1-e^ {-(c_2L^{\theta - \beta})}
$$
for $L$ sufficiently large.
\end{lemma}

Similarly we arrive at
\begin{proposition}
\label{cprect} Given finite positive $K$, there exists a constant $c_3 = c_3(K) \in (0 , \infty )$ so that for $l$ sufficiently large if $\xi^R$ is a contact process restricted to
$[0,Kl] \times [0,l]$ so that $\xi^R_0 $ has no vacant subinterval of $[0,Kl]$ of length $v$, then
$$
P( \xi^R_l \equiv 0 \mbox{ on an interval } I \subset [0,Kl] \mbox{ with } |I| \geq u] \ \leq \
\frac{Kl}{c_3} ( e^{-c_3u} + e^{-c_3l/v}).
$$
\end{proposition}

In particular we have the following.

\begin{corollary}  \label{good1}
There exists a constant $c_4 \in (0,+\infty)$ so that for all $n$ large,  if $x \in (-n^2 2^n ,n^2 2^n ) $
 and $\xi_t $ is a configuration with no $n^{3/2} $
vacant intervals on $[-2n^2 2^n ,2n^2 2^n ]$, then outside a set of
probability at most $n^9 e^{-c_4 \log^{3/2}(n)}$, the configuration $\xi_{t+ n^4}$ has no gap
of size $\log^{3/2}(n)$ within $n^{9}$ of $x$.
\end{corollary}

Simple large deviations estimates for the rightmost particle for a one-sided initial configuration
give the following:

\begin{lemma}
\label{cpspeed}
There exists a constant $c_5 \in (0, \infty ) $ so that for a given Harris system, the chance that there
is a path from $(- \infty,0) \times (0,T) $ to $(RT, \infty) \times(0,T)$ is less than $e^{-c_5RT} $ for
all $R > \frac{1}{c_5}, \ T > 1$.
\end{lemma}

\section{An approximate equilibrium}
\label{sec-ac}

\noindent Consider a Harris system on $\Z \times (- \infty, 0)$ and define $\xi'$ and $\xi$ as follows:
\vspace{0.1 cm}

\noindent $\xi'(x) = 1$ if and only if ${(\hat{\xi}^{x}_{s})}_{s \geq 0}$ the dual
restricted to some space time region $R_{x}$ satisfies some condition $C_{x}$.
\vspace{0.2cm}

\noindent $\xi(x) = 1$ if and only if  $\hat{\xi}^{x}$ survives forever. In particular $\xi$ is in equilibrium $\bar \nu$.
\vspace{0.3cm}

\noindent  Writing $p_{x}$ for the probability that $\xi'(x)\neq\xi(x)$, it is clear that if
$\sum_{x} \ p_x \leq \frac{1}{2}$, then $ {\xi}^\prime$ has a law equal
to the equilibrium distribution conditioned on an event of reasonable probability.

\vspace{0.3cm}
We now get to define $R_x$ and $C_x $ adapted to scale $2^n$.  We first fix 
$h_1$ the positive constant of
Lemma \ref{lemreg}.

A) For $\vert x \vert \leq n^9:$

$R_{x}=\Z \times (-t(n),0)$ where  $t(n) =log^4(n)/2$ 
 and $C_{x}$ is the event that at time $t(n)$, $\hat{\xi}^{x} \cap [-n^9, n^9 ]$ has size at least
 $\geq h_1 t(n)$. 

\vspace{0.2cm}

B) For $\vert x \vert >n^9:$

$R_{x}=\Z \times (- \infty, 0)$ and $C_{x}$ the condition of surviving forever.

\vspace{0.3cm}
To verify the condition
$$
\sum_x p_x \  <  \ 1/2,
$$
for $n$ large, we first note that the summands for $\vert x \vert > n^9  $ are zero.  Secondly we
have by (i) in Proposition \ref{cpinterval}, and taking $c_1$ the positive constant in that statement:
$$
\sum_{|x|\leq n^9} P(\xi'(x) =1, \xi (x) = 0) \leq \ \frac{1}{c_1}\left( (2n^9+1) e^{-c_1 h_1 log^4 (n)/2}  \right)
$$
which converges to zero as $n$ tends to infinity.  

For the term $\sum_x P(\xi'(x) =0, \xi (x) = 1)$  summed over $\vert x  \vert \leq n^9$,
we apply  Lemma \ref{lemreg} to get the required bounds.

We now alter this definition for  $\vert x \vert > n^9 $.  The objective is to define a
configuration which is essentially the same as above but which is independent of certain
rectangles of the Harris system.  The ``cost" of losing the global closeness to equilibrium by changing
the values far away is small compared to the independence gained.  We replace condition B) with \\

$\text{B}^\prime $) for $\vert x \vert > n^9  $, we set $\xi^\prime (x) =1$.

It is to be noted that with this amended definition the configuration
$\xi^\prime $ is independent of the Harris system on
$\Z \times (- \infty, -t(n))$

We let $\nu (=\nu(n))$ be the distribution of the configuration $\xi^\prime$ with rules given in A) and $\text{B}^\prime )$ above,
conditioned on the event that for all $x \in [-n^9, n^9]$, $\xi'(x)=1$ whenever $\hat \xi^ x$ survives till time  $t(n)$.

\section{A regeneration time}
\label{sec-reg}

The purpose of this section is to describe a regeneration time $\sigma = \sigma(n,T) $ associated to a
space and time scale $2^n$ and a stopping time $T$ (also called $n$ order regeneration time). We remark 
that stopping times in this paper will always refer to the natural filtration of our Harris system (plus 
some auxiliary random variables). In this sense, also the regeneration time will be a stopping time 
occurring after the stopping time $T$. The construction will be such that at time $\sigma $ a
random configuration $\xi ^\prime_ \sigma$ will be produced so that\\
(i) $\xi ^\prime_ \sigma$ has distribution $\nu$ ($= \nu (n)$ as in Section \ref{sec-ac}) relative to $X_\sigma$, \\
(ii) with very high probability $ \xi ^\prime_ \sigma (x)\ = \ \xi _ \sigma (x) $ for $ \vert x- X_ \sigma \vert \leq n^9  $.

{\it Remark:} Of course given $\xi ^\prime_ \sigma (x)\equiv 1$ for $ \vert x- X_ \sigma \vert > n^9 $, we cannot have
$ \xi ^\prime_ \sigma \ = \ \xi _ \sigma $.  The idea is that in subsequent evolution of a dynamic
random walk $(\xi^\prime,X^\prime )$
with $  X ^\prime_ \sigma \ = \ X _ \sigma$ we have $X ^\prime_s\ = \ X  _s$
with very high probability.  See Lemma \ref{noencroach}.

We will suppose $n$ is fixed and drop it from notation for our regeneration time $\sigma=\sigma(n)$.
We also translate our Harris system temporally by $T$ and spatially by $X_T$ so in the following we take
$(T,X_T) \ = \ (0,0)$.

The time $\sigma $ is obtained via a series of {\it runs}. 
Each run will probably be aborted before completion but if it doesn't then,
as far as evolution on a scale of $2^n $ is concerned, the process will start from a given distribution
(which will weakly depend upon $n$).  In the following, $t$ could be any time in $(0,2^{2n}) $,
but see the discussion at the end of the section on extending it to times in $[0,2^{Kn}]$ for $K$ large but fixed.
 We begin a run at time $t$ (for the first run $t=0$) by considering the joint $(\xi,X)$ process on time
interval $(t,t+n^4 + \log^4(n))$.  If the run is aborted, then we try a subsequent ``run" at time $t+n^4 + \log^4(n)$ 
and so on until a complete run is obtained. 

A run consists of at most five stages. A (complete) run will produce $\sigma = t+n^4 + \log^4(n)$, either  
if the first stage is a failure or if all five stages succeed, in which case we say that the run is successful. 
The latter case will be good from the point of view of (ii) above while the first case will be bad (but mercifully of 
small probability). If the run is {\it successful}, then the distribution of
$  \xi_{t+n^4 + \log^4(n)  }$ shifted by $X_{t+n^4 + \log^4(n) }$ will be $\nu$ at least on interval $(-n^9 , n^9 )$.

As we shall now see: the first stage will succeed with very large probability. The second stage will succeed with
probability of order $e^{-M^{\prime}\log^4(n)}$.  The next two will succeed with high probability (for $n$ large),
the fifth with a reasonable probability, given the success of the previous stages.

$\bullet$ The first stage consists simply of seeing if there is a vacant gap of size $n^{3/2} $ for $\xi_t$ on $(X_t-2n^2 2^n, X_t+2n^2 2^n)$.
If so we conclude the run and designate $\sigma = t+n^4 + \log^4(n) $.  We put
$X^\prime_\sigma$ equal to $X_\sigma$ and $\xi ^\prime_\sigma$ to be a random distribution independent of the natural
Harris system for $(\xi,X)$ so that shifted by $X_\sigma, \ \xi^\prime_\sigma$ has distribution $\nu $. 
Technically this gives a complete run (since it has established $\sigma $) but of course this case has severed any link
between $\xi $ and $\xi^\prime$ and will be treated as a ``disaster".  It is however easy to see that the chance this occurs
for $t$ in $[0,2^{Kn}]$ (fixed $K$) is bounded by $e^{-cn^{3/2}}$
for some universal $c$ (see Lemma \ref{nogaps} below).   Thus the contribution to the various integrals considered in later sections
will be negligible. If there are no $n^{3/2} $ gaps we describe the first stage of the run as a success.

$\bullet$ The second stage is a success if (recalling the notation set in the introduction) we have \\
1) $N^X(t+n^4)-N^X(t) \ < \ M n^4$ and \\
2) $N^X(t+n^4+ \log^4(n))-N^X(t+n^4) \ = \ 0$.

We remark that the first condition is satisfied with probability tending to one as $n$ becomes large, while the second condition has
probability exactly $e^{-M^{\prime}\log^4 (n)}$.  The first  condition implies that whatever the contact process might be,
$X$ moves less than $M n^4$ on the time interval $(t,t+n^4 )$ 
and is constant on the time interval
$(t+n^4,t+n^4+ \log^4(n) )$.  As with subsequent stages, if this is a failure
we let the process run up until time $t+n^4+ \log^4(n)  $ in order to regain the Markov property.

$\bullet$ Given the second stage is a success, we pass to the next, and require that on the interval
\footnote{When the notation would be too clumsy, we write $X(t)$ instead of $X_t$.}
$(X(t+n^4)-n^9,X(t+n^4)+n^9)$ there be no gaps of size $log^{3/2}(n)$ for $\xi_{t+n^4}$.

We note here that as $\xi_t$ has no $n^{3/2}$ gaps the chance of this event, given also a successful second
stage, is close to one by Corollary \ref{good1}.

$\bullet$ For the fourth stage we construct $\xi'_{t+n^4+\log^4 (n) } $
according to the $n$ level specifications from the given Harris system shifted spatially
by $X(t+n^4+ \log^4(n))$ and temporally by $t_1=t+n^4+\log^4 (n)$:\\
 For $|x| \leq n^9$, $ \xi' _{t+n^4+\log^4 (n)}(X(t+n^4+ \log^4(n))+x) = 1$ if and only if
the dual $\hat\xi^{x+X(t+n^4+ \log^4(n)),t_1} $ survives for time
$t(n)=\log^4(n)/2$ satisfying condition $C_x$ (Sec. \ref{sec-ac}) suitably displaced.

\noindent For $|x|>n^9$ we set $ \xi'_{t+n^4+\log^4 (n)}(X(t+n^4+ \log^4(n))+x) = 1$.
\\
The fourth stage is successful if
$$
\xi'_{t+n^4+\log^4 (n) } (x) \ge \tilde\xi_{t+n^4+\log^4 (n) } (x), \quad {\text{ for all}} \quad x
$$
where $\tilde\xi_{t+n^4+\log^4 (n)}$ is given by
$\tilde\xi_{t+n^4+\log^4 (n)}(x)= 1$ if and only if the dual\\
$\hat\xi^{x,t_1} $ survives for time
$t(n)=\log^4(n)/2$ (or equivalently $\tilde\xi_{\cdot}$ is the contact process defined on the Harris system from time
$t+n^4+\log^4 (n)/2$,  starting with full occupancy). By Lemma \ref{lemreg},
the probability of success at the fourth step tends to one as $n$ tends to infinity.
Indeed, it amounts to prove that for all $|x|\le n^9$ one
has $ \xi' _{t+n^4+\log^4 (n)}(X(t+n^4+ \log^4(n))+x) = 1$ whenever $\hat\xi^{x+X(t+n^4+ \log^4(n)),t_1}=1$, and
the calculation is essentially given in Section 3.
It is to be noted that this condition relies on a Harris system disjoint from (and independent of) the Harris
systems observed in stage 2.

$\bullet$ Finally for the fifth stage we note that, provided that the requisite stages have been successfully
passed, the conditional chance that for every $x$ such that $|X_{t+n^4}-x|\leq n^9 $ and $\xi' _{t+n^4+\log^4 (n) } (x)=1$ one has
$$
\xi' _{t+n^4+\log^4 (n) } (x)  \ = \ \xi _{t+n^4+\log^4 (n) } (x)
$$
is at least $3/4$ for $n$ large, as an easy calculation using Lemma \ref{hit} shows. This conditional probability will
depend on the random configuration $\xi' _{t+n^4+\log^4 (n) } $ as well as the configuration
$\xi_{t+n^4}$. Let us denote it by $p(\xi' _{t+n^4+\log^4 (n) }, \xi_{t+n^4})$.   Having introduced an auxiliary independent uniform
random variable $U$ associated to the ``run" (enlarging the probability space if necessary), we then say that the
run is (globally) a success if
$$
\xi' _{t+n^4+\log^4 (n) } (x)  \ = \ \xi _{t+n^4+\log^4 (n)} (x)
$$
and $U \ \leq \ \frac{3/4}{p(\xi' _{t+n^4+\log^4 (n) }  , \xi_{t+n^4})}$.

Using this randomization procedure we see that, conditionally on success,
the distribution of $\xi_{t+n^4+\log^4 (n) } $  shifted by  $X_{t+n^4+\log^4 (n) }$ and
restricted to the interval $[-n^9,n^9]$ coincides with $\nu$ restricted to $[-n^9, n^9]$.

{\it Notation:} For a stopping time $T$, we let $\sigma_T \ =\ \sigma(n,T)$ denote the end time of the first successful
run after beginning the runs at time $T$:\\
$\sigma_T \ = \ \inf \{ T + k(n^4+\log^4 (n) ): $ a successful run is initiated at time
$T + (k-1)(n^4+\log^4 (n) )\}$.  We say $ \sigma_T$ is the $T$ regeneration. We say that
a {\it disaster} occurs at $\sigma_T$ if $\xi'_{\sigma_T} (x) \neq \xi_{\sigma_T} (x) $
for some $x$ within $n^9$ of $X_{\sigma_T}$.
The arguments given above  imply the following.

\begin{lemma}\label{reg-lemma}
There exists a positive constant $c$ so that the probability of a successful run is at least $\frac{c}{n} e^{-M^\prime log^4(n)}$.
\end{lemma}

\begin{proposition} \label{gencontrol}
There exists a constant $c_6 \ \in \ (0, \infty )$ so that for $n $ large and any stopping time, $T$, for the
filtration $(\mathcal F_t)$ determined by the $(\xi,X)$ process, 
$\sigma_T $, the first time for a successful run starting at $T$, satisfies
$$
P(\sigma_T > T+n^8e^{M^\prime \log^4 (n)} |\cF_T) \ < \ e^{-c_6n^3}.
$$
\end{proposition}
\begin{proof}
Since each run takes an interval of time of length at most $n^4 +\log^4(n)\le 2n^4$, the proof follows at once from Lemma \ref{reg-lemma} by repeated application of the Markov property.
\end{proof}




The next result is crucial for our regeneration arguments as it implies that replacing
at regeneration times our configuration $\xi_ \sigma $ by a regenerated $\xi^ \prime_ \sigma $ 
does not change the
evolution of $X$, thus facilitating an i.i.d. structure for $X$.
\begin{lemma} \label{noencroach}
Consider two dynamic random walks $(\xi,X)$ and $(\xi^\prime,X^\prime )$ run with the same Harris system
and so that \\
(i)  $X_0 = X^\prime _0 $\\
(ii) $\xi_0 (x) =   \xi^\prime_0 (x) $ for $\vert x-X_0 \vert \leq n^9$,
elsewhere $\xi_0 (x) \leq   \xi^\prime_0 (x)$.

Then, for a suitable positive constant $c_8$, outside of probability $e^{-c_8n^2}$ we have either\\
a) $\xi_0$ has an $n^{3/2} $ gap within $2^{4n} $ of $X_0$, or\\
b) $X_{n^4} = X^\prime_{n^4} $  and $\xi_{n^4} (x) =   \xi^\prime_{n^4} (x) $
for $\vert x-X_{n^4} \vert \leq 2^{3n}$.
\end{lemma}

\begin{proof}
It is only necessary to show the inequality holding for $n$
sufficiently large, so in the following we will take $n$ to be large
enough that a finite number of asymptotic inequalities hold. It suffices
to consider  initial configurations $(\xi_0,X_0)$ and $(\xi'_0, X'_0)$ as in {\it (i)} and {\it (ii)}  such that
$\xi_0$ has no $n^{3/2} $ gaps within $2^{4n} $ of $X_0$, and to find  two ``bad" events $B_1$ and $B_2 $,
with appropriately small probabilities conditioned on such $\xi_0$, and such that if neither $B_1$ nor $B_2$ occur, then
\begin{equation}
\xi_{n^4} (x) \ = \ \xi^ \prime _{n^4} (x) \ \text{ for all } x \in [X_{n^4}-2^{3n},X_{n^4}+2^{3n}]
\label{coinc-1}
\end{equation}
and
\begin{equation}
\label{coinc-2}
X_{n^4} \ = \ X^ \prime _{n^4}.
\end{equation}


The first event, $B_1$, is simply taken as $ \cup _ {x \in [X_0-2^{4n}/ 2, X_0+2^{4n}/ 2 ]} B_1^x $, where
\begin{equation}
\label{bad1}
B_1^x=\left\{ \hat \xi ^{x,n^4}_{n^4} \ne \emptyset  ,\; \hat \xi
^{x,n^4}_{n^4} \cap \xi_0 \cap [X_0-2^{4n}, X_0+2^{4n}] = \emptyset
\right\}.
\end{equation}

We clearly have $B_1^x \subset C_1^x \cup C_2^x $, where
$$C_1^x \ = \ \left\{ \hat \xi ^{x,n^4}_{n^4} \ne \emptyset ,\; \vert  \hat \xi ^{x,n^4}_{n^4} \cap (x,x+n^4 /2)  \vert \leq h_1n^4 /2\right\} $$
and
$$C_2^x \ = \ \left\{\vert \hat \xi ^{x,n^4}_{n^4} \cap (x,x+n^4 /2)  \vert > h_1n^4 /2,\, \hat \xi ^{x,n^4}_{n^4} \cap \xi_0 = \emptyset \right \}.$$

The event $C_1^x$ is independent of $\xi_0$, and by Lemma \ref{lemreg}, $P(C_1^x) \ \leq \ \frac{1}{h_1} e^{-h_1n^4
/2} $.  We now remark that for $n $ large $(x,x+n^4/2) \subset [X_0-2^{4n},
X_0-2^{4n}] $ for all $x$ in the range of interest, and so $\xi _0 $ will have no $n^{3/2}$ vacant
intervals in this interval. Also by Lemma \ref{hit}, $P^{\xi_0}(C_2^x) \ \leq \  e^{-c_2h_1n^{5 /2}}$, uniformly over $x$ and $\xi_0$ under the given condition of no gaps, where $P^{\xi_0}$ refers to the conditional probability given $\xi_0$. Therefore uniformly for all such $\xi_0$
$$
P^{\xi_0} (B_1) \ \leq \ (2^{4n} + 1) \left( e^{-c_2h_1n^{5 /2}} +
\frac{1}{h_1} e^{-h_1n^4 /2} \right)
$$
which is less than $\frac{1}{2} e^{-n^2}$ for $n $ large.

As for the second bad event, its complement $B_2^c$ is given by
\begin{equation}
\label{bad2}
B_2^c=\left\{X^\prime_{n^4} \ = \ X_{n^4} \ \in \ [X_0-n^5 , X_0 + n^5]\right\}.
\end{equation}

Before estimating $P^{\xi_0}(B_2)$ uniformly over $\xi_0$ as above, 
notice that the desired properties  \eqref{coinc-1}, \eqref{coinc-2} hold on $B_1^c \cap B_2^c$. 
This is automatic for \eqref{coinc-2}. On the other hand, \eqref{coinc-1} follows from \eqref{bad1} and \eqref{bad2} once we
take {\it (ii)} into account. 

First we have by simple tail estimates for Poisson random variables,
$$
P^{\xi_0}( \sup_{0\le s \le n^4} (|X_s-X_0|  \vee |X^\prime_s-X^\prime_0|) > n^5) \le e^{-cn^5}
$$
for some strictly positive $c$ depending on $M'$ but not on $n$. So it
remains to argue that $X_{n^4}$ and $X^\prime_{n^4}$ must be equal with
very large probability. To do this it will suffice to show that
(outside a set of very small probability),
\begin{equation}
\label{coinc-3}
\xi_s(x) \ = \ \xi^ \prime _s(x) \ \text{ for all } 0 \leq  s \leq n^4, \text { for all } x \in [X_0-n^5-r_0, X_0+n^5+r_0]
\end{equation}
where $r_0$ was defined in the first paragraph.
We divide up $  [X_0-n^9 , X_0 + n^9)$ into disjoint intervals $
\{I_i\}_{i \in J}$ of cardinality $n^5$. 
Let the collection of indices of intervals entirely to the left of $[X_0-n^5-r_0 , X_0 + n^5+r_0]$ be $J_a$,
while $J_b$ denotes the collection of indices for intervals entirely to the right.
For any $i \in J$, let  $D_i $ be the event that there is a path from
$(x,0)$ to some $(y,n^4)$ entirely contained in space time rectangle
$I_i \times [0,n^4]$ for some $x$ with $\xi_0(x) = \xi_0^\prime (x) =
1$.  We note first that the events $D_i $ are conditionally
independent given $\xi_0 $ and that by our restriction on the size
of gaps for $\xi_0 $ we have that given $ \xi_0 $ each $D_i $ has a
conditional probability bounded away from zero in a way that depends
on $\lambda$ but not on $n$.   From this it follows that uniformly over $\xi_0$
without $n^{3/2}$ gaps as above, we have
$$
P^{\xi_0}\left( \cap _{i \in J_a} D_i^c \right) \ + \  P^{\xi_0}\left( \cap _{i \in
J_b} D_i^c \right) \leq  \ e^{-c^ \prime n^4 }.
$$
But \eqref{coinc-3} holds at once on the set $\cup_{i \in J_a} D_i \cap \cup _{i \in
J_b} D_i$, and we are done.

\end{proof}

\section{Existence of normalizing constants}
\label{sec-5}
\noindent In this section we wish to use our coupling time to establish the existence
of $\mu$ and $\alpha$ so that as $n \to \infty$

$$E\left( \frac{X_{2^{n}}}{2^{n}} \right) \to \mu
\ \mbox{ and } \frac{1}{2^{n}} E {(X_{2^{n}} -2^{n} \mu)}^{2} \to \alpha^{2}.$$

We first state the following general result, which is shown through basic techniques:
\begin{lemma}  \label{nogaps}
There exists a constant $c_{9} > 0$ so that, if the contact process $\xi $ is in equilibrium,
then for all $n$ large 
$$
P( \exists t \leq 2^{3n}, \ \vert x \vert  \leq 2.2^{4n} \mbox { so that } \xi_t \equiv 0
\mbox { on } (x, x+ n^{3/2} ) \ \leq \ e^{-c_{9}n^{3/2}}.
$$
\end{lemma}

Considering $(\xi', X')$ associated to a regeneration time $\sigma=\sigma(n,T)$ for $T=0$ as in
the previous section, and $\nu=\nu(n)$ the measure defined in Section \ref{sec-ac} we will need:

\begin{lemma}  \label{base}
There exists a constant $c_{10} < \infty$ so that for all $n$,
$$
\big \vert E\left( \frac{X_{2^{n}}}{2^{n}} \right)  \ - \  E^{\nu(n)}\left( \frac{X_{2^{n}}}{2^{n}} \right)
\big \vert \quad < \quad \frac{c_{10}n^8 e^{M^{\prime}\log^4 (n)}}{2^{n}} .
$$
\end{lemma}
\begin{proof}
Let $\sigma = \sigma(n,0) $ be the $n$ order regeneration time for time $0$ and  $X'$ the dynamic random walk resulting
from $ \sigma $. Let $D$ denote the event
that either\\
(i) $\sigma > n^8e^{M^{\prime}\log^4 (n)} $ or \\
(ii) $X_{\sigma + 2^n} - X_\sigma \ \ne \ X^\prime _{\sigma + 2^n} - X^\prime_\sigma$. \\
By Lemmas \ref{noencroach} and  \ref{nogaps} and Proposition \ref{gencontrol} we have
that $P(D) < e^{-c_9n^{3/2}} +  e^{-c_8n^{2}}+ e^{-c_6n^{3}}$. But
$$
\big \vert  E\left(X_{2^{n}}\right)  \ - \  E^{\nu(n)}\left( X_{2^{n}}\right)\big \vert \ = \
\big \vert  E\left(X_{2^{n}}\right)  \ - \ E\left(X^\prime_{\sigma +2^{n}}-X^\prime_{\sigma}\right)\big \vert
$$
$$
\le E\left( |X_{2^{n}} - (X_{\sigma +2^{n}}-X_{\sigma})|I_{D^c}\right) +
E \left(\big \vert X_{2^{n}}  \ - \ \left(X^\prime_{\sigma +2^{n}}-X^\prime_{\sigma}\right)\big \vert I_D\right).
$$
The first term is bounded by $E\left(N^X(n^8 e^{M\log^4(n)})+ N^X(2^{n}+n^8 e^{M\log^4(n)})-N^X(2^{n})\right)$,
which is $2M^{\prime} n^8 e^{M^{\prime}\log^4(n)}$. For the term containing $I_D$ we use the Cauchy Schwarz inequality to conclude
the proof.
\end{proof}

\begin{lemma}  \label{exp}
There exists a constant $c_{11} < \infty$ so that for all $n$,
$$
\big \vert E\left( \frac{X_{2^{n}}}{2^{n}} \right)  \ - 
\  E\left( \frac{X_{2^{n+1}}}{2^{n+1}} \right) \big \vert \quad <
\quad \frac{c_{11}n^8e^{M^{\prime}\log^4 (n)} }{2^{n}} .
$$
\end{lemma}


\begin{proof}
\noindent We begin by now taking $\theta=\sigma(n,2^n)$ to be the regeneration time after $2 ^n$, 
and argue as in the proof of 
Lemma \ref{base}, where we now take the event $D$ that either\\
(i) $\theta -2^n> n^8e^{M^{\prime}\log^4 (n)} $ or \\
(ii) $X_{\theta+ 2^n} - X_\theta \ \ne \ X^\prime _{\theta + 2^n} - X^\prime_\theta$.

\noindent As in the proof of Lemma \ref{base}, $D$ has a very small probability. 
Proceeding them as in that proof, and since $E(X^{\prime}_{\theta + 2^n} - X^\prime_\theta) = 
E^{\nu(n)}(X_{2^n})$ 
we can write
\begin{eqnarray*}
EX_{2^{n+1}} &=&E X_{2^{n}} + E(X_{\theta} -X_{2^{n}}) + E(X_{\theta + 2^{n}} - X_{\theta}) + E(X_{2^{n+1}}- X_{\theta + 2^{n}})\\
&=& 2 E X_{2^{n}} + O (n^{8}e^{M^{\prime}\log^4 (n)}) + E(X_{\theta}-X_{2^{n}}) + E(X_{2^{n+1}}- X_{\theta + 2^{n}}) 
\\
&=& 2EX_{2^{n}}+ O (n^{8}e^{M^{\prime}\log^4 (n)})
\end{eqnarray*}
where in the second equality, we have also used Lemma \ref{base}. 
\noindent From this the lemma follows.
\end{proof}

This immediately begets
\begin{corollary} \label{comu}
There exists $\mu \in (-\infty, \infty )  $ so that
$$
\lim_{n \rightarrow \infty}  \frac{E(X_{2^n})}{2^n} \quad = \quad \mu .
$$
\end{corollary}
\vspace{1cm}
{\it Remark: } It is not difficult to see that $E(X_t/t)$ converges to $\mu$, see for instance the
proof of Lemma \ref{l2bound}.

\noindent We now look for  a bound for $E {\left(\frac{X_{2^{2n}}}{2^{2n}} - \mu\right)}^{2}   $.

\noindent As we have seen $\mu =\lim_{n \to \infty} \frac{E \left( X_{2^{n}} \right)}{2^{n}}$ exists.
Furthermore by Lemma \ref{base} for \\
$\mu_{n} = \frac{E^{\nu(n)}\left( X_{2^{n}}\right)}{2^{n}} $ we have
$\vert \mu_{n} - \mu \vert  \leq \frac{c_{10} n^{8}e^{M^{\prime}\log^4 (n)}}{2^{n}}.$

\vspace{0.3cm}

Given a dynamic random walk $(\xi, X) $ and a scale $n$, we define a sequence of renewal points
$\beta_i, \ i \geq 1$ as follows: let $ \beta_1 \ = \ \sigma (0)$, the regeneration time for the
stopping time $0$. Subsequently for $i \geq 1$, we define $\beta_{i+1}$ so that $\beta_{i+1} - \beta_i$ is the
regeneration time for the stopping time $2^n$ for the dynamic random walk $(\xi^{i}, X^{i}) $ so that \\
(i) $(\xi^{i}, X^{i}) $ is generated by the Harris system temporally shifted by $\beta_i $;\\
(ii) for $\vert x- X^i_{\beta_{i+1}-\beta_i} \vert  \leq n^9 , \  \ \xi^{i+1} _0(x) =  (\xi^{i}) ^\prime  _{\beta_{i+1} - \beta_{i
} }(x) $; elsewhere  $\xi^{i+1} _0(x) = 1$;\\
(iii) $ X^{i+1}_0 \ = \ X^i_{\beta_{i+1}- \beta_i} $.

 We wish to only deal with ``good" $i$, where we say that $i$ is good if all $0 \leq j <i$ are good and if \\
(a) $ \beta_{i+1} - \beta_i \leq \ 2^n + n^8e^{M^{\prime}\log^4 (n)}$\\
(b) $N^X (\beta_{i}) - N^X (\beta_{i-1})  < 2n^2 2^n $, where for notational completeness, we take $\beta_0=0$ and
$(\xi^{0}, X^{0}) $ to be the original dynamic random walk $(\xi, X) $.

Setting $S = \inf \{i: i \mbox { is not good} \}$, we define the random variables $Z_i, \ i \geq 1$ by\\
$\bullet$ for $j < S, \ Z_j = X^{j}_{\beta_{j+1}-\beta_{j}} - X^{j}_0$,\\
$\bullet$ for $j \geq S, \ Z_j$ are taken from an independent i.i.d. sequence of random variables with the distribution that of $Z_1$
conditioned on $1$ being good.

We note that unless a disaster occurs at some stage $\beta_j$
(i.e. $\xi^{j}_0(x) \neq \xi^{j-1}_{\beta_j- \beta_{j-1}}(x)$ for some $x$ within $n^9$ of
$X^{j-1}_{\beta_j- \beta_{j-1}} $) we have for each $i$, $\xi^ {i}_0 \geq \xi_{\beta_{i}}$.
Thus we have via Proposition \ref{gencontrol} and Lemmas \ref {noencroach}  and \ref {nogaps}, and simple
estimate with the Poisson distribution, that outside a set of probability 
$2^n (e^{-c_{9}n^{3/2}} +e^{-c_8n^2}+e^{-cn^2}+ e^{-c_6n^3})$, for some $c>0$ that depends on 
$M^{\prime}$, all $i \leq 2^n $ are good and for such $i$ we have
$$
X_{\beta_i}- X_{\beta_1}  \ = \ \sum_{k=1}^{i-1}Z_k.
$$

Take $R$ to be the integer part of $\frac{2^{2n}}{2^n + n^8e^{M^{\prime} \log^4 (n)}} - 2$ and let us define
$$
Y_{2^{2n}}  = X_{\beta_1} \ + \ \sum_{j=1}^R Z_j \ + \ (X_{2^{2n}} - X_{\beta_{R+1}} ) \ \equiv \sum_{j=1}^R Z_j +F_n.
$$
As noted, outside probability $2^n (e^{-c_{9}n^{3/2}} +e^{-c_8n^2}+e^{-cn^2}+ e^{-c_6n^3}), \ Y_{2^{2n}}  =  X_{2^{2n}}$.

Now, by techniques already employed for Lemma \ref{base}, it is easy to see that for suitable universal $c_{12}$, we have
$$
\vert E(Z_1) - 2^n \mu    \vert \leq \ c_{12} n^8e^{M^{\prime}\log^4 (n)}
$$
and so we may write (with $Z_j^\prime$ equal to $Z_j $ minus its expectation)
$$
Y_{2^{2n}}-2^{2n} \mu   =  \ \sum_{j=1}^R Z_j ^\prime \ + \ F_n^\prime.
$$
where $E((F_n^\prime)^2) \leq c n^{16}e^{2M^{\prime}\log^4 (n)}2^{2n}$ for universal $c$. 
From this and the obvious bound $E(Z_j^2) \leq \tilde c 2^{2n}$
we obtain, for some universal $C$, that
$$
E(Y_{2^{2n}}-2^{2n} \mu)^2 \ \leq \ C 2^{3n}.
$$

We finally use the elementary identity
$$
E(X_{2^{2n}}-2^{2n} \mu)^2 \ = \ E(Y_{2^{2n}}-2^{2n} \mu)^2 \ + \
E\left(((X_{2^{2n}}-2^{2n} \mu)^2-(Y_{2^{2n}}-2^{2n} \mu)^2) I_{X_{2^{2n}} \ne Y_{2^{2n}}} \right)
$$
and Cauchy Schwarz to conclude

\vspace{0.3cm}

\begin{lemma}  \label{basic}
There exists universal constant $c_{13}$ so that for all positive integer $n$,
$E({(X_{2^{2n}}- 2^{2n}\mu)}^{2}) \leq c_{13} 2^{3n}$.
\end{lemma}

\noindent We can now prove

\begin{proposition} \label{var}
\noindent For $X$ as defined above, there exists  $\alpha \in (0, \infty)$ so that as $n \to \infty$
$$2^{n} E {\left( \frac{X_{2^{n}}}{2^{n}}- \mu \right)}^{2} \to \alpha^{2}.$$
\end{proposition}

\begin{proof}

\noindent The proof that the limit is strictly positive is given below in Proposition \ref{positif}. 
\noindent
For the existence, we write as before
$X_{2^{n+1}}-2^{n+1} \mu$ as $(X_{2^{n}}-2^{n} \mu ) + Y_{1} + Z_{1} + Y_{2}$, where $Y_{1}$ is the increment
of $X$ over time $[2^{n}, \sigma]$ with $\sigma=\sigma(n,2^n)$  being the time of the regeneration after time $2^n$.
$Z_{1} = X^\prime_{\sigma +2^{n}}- X^\prime_{\sigma} -E (X^\prime_{\sigma +2^{n}} - X^\prime_{\sigma})$ and $Y_{2}$
is defined via the above equality. Then we have
\begin{eqnarray*}
E(( {X_{2^{n+1}}-2^{n+1} \mu)}^{2})&=&E((X_{2^{n}}-2^{n} \mu)^{2})  + E (Z_{1}^{2}) + 2 E ((X_{2^{n}} -2^{n} \mu) Z_{1})\\
&\quad+& E({(Y_{1} + Y_{2})}^{2})+ 2E((Y_{1} + Y_{2}) Z_{1}) \\&\quad+& 2E((Y_{1} + Y_{2}) (X_{2^{n}}-2^{n} \mu)).
\end{eqnarray*}
\noindent By our choice of $Z_{1}$ we have $2E((X_{2^{n}} -2^{n} \mu) Z_{1}) = 0$,
while by Cauchy Schwarz and Lemma \ref{exp} we have, for $n$ large, (and some finite $K$ not depending on $n$)
$$
\vert  E({(Y_{1} + Y_{2})}^{2}) + 2E((Y_{1} + Y_{2}) Z_{1}) + 2E((Y_{1} + Y_{2}) (X_{2^{n}}-2^{n} \mu))\vert  \leq K2^{\frac{3n}{4}}n^{8}e^{M^{\prime}\log^4 (n)}.
$$
\noindent It simply remains to check that (increasing $K$ if necessary)
$\vert E Z_{1}^{2} - E ( {(X_{2^{n}}-2^{n} \mu)}^{2})\vert \leq Kn^{8} e^{M^{\prime}\log^4 (n)}2^{\frac{3n}{4}}$
to see that
\begin{equation}
 \label{eq-var}
\vert E( {(X_{2^{n+1}}-2^{n+1} \mu)}^{2})-2 E ( {(X_{2^{n}}-2^{n} \mu)}^{2})\vert \leq c 2^{\frac{3n}{4}}n^{8}e^{M^{\prime}\log^4 (n)},
\end{equation}
from which we obtain the existence of limit of $2^{-n} E ( {(X_{2^{n}}-2^{n} \mu)}^{2})$.
\end{proof}

\vspace{0.3cm}

\noindent We now adapt the previous argument to give bounds on $E ( {(X_{2^{n+1}}-2^{n+1} \mu)}^{4})$.

\noindent As before we write
$$
X_{2^{n+1}}-2^{n+1} \mu=
X_{2^{n}} -2^{n} \mu + Y_{1} + Z_{1} + Y_{2} = X_{2^{n}}-2^{n} \mu + Z_{1} + Y.
$$
So we can write ${(X_{2^{n+1}}-2^{n+1} \mu)}^{4}$ as
$$
{(X_{2^{n}}-2^{n} \mu)}^{4} + Z_{1}^{4} + 6 {(X_{2^{n}}-2^{n} \mu)}^{2} Z_{1}^{2} + 4 Z_{1}^{3}(X_{2^{n}}-2^{n} \mu)+
4 Z_{1} {(X_{2^{n}}-2^{n} \mu)}^{3} + W,
$$
where the random variable $ W$ is defined by the above equality.
Now $E(Z_{1} {(X_{2^{n}}-2^{n} \mu)}^{3})=0$ and
$E ({(X_{2^{n}}-2^{n} \mu)}^{2}Z_{1}^{2}) = E ({(X_{2^{n}}-2^{n} \mu)}^{2} )E (Z_{1}^{2})=2^{2n} \alpha ^ 4 (1+O(1)) $,
while $ \vert E (Z_{1}^{3}(X_{2^{n}}-2^{n} \mu)) \vert  =
\vert E Z_{1}^{3} \vert  \vert E (X_{2^{n}}-2^{n} \mu)\vert \ \leq \ {(E Z_{1}^{4})}^{\frac{3}{4}}
\vert E (X_{2^{n}}-2^{n} \mu) \vert \leq \ {(E Z_{1}^{4})}^{\frac{3}{4}} K n^8 e^{M^{\prime}\log^4(n)}$. On the
other hand,
$$
E(W)
\ = \ E( Y^4) + 6E( Y^2{V}^{2}) + 4E( Y^3 V) +
 4E(Y V^{3} )
$$
for $V \ = \  X_{2^{n}}-2^{n} \mu + Z_1 $.  Using  Holder's inequality, we see that
$$
E(W)\leq K {(EY^{4})}^{\frac{1}{4}}
\left(  (E (X_{2^n}- {2^n}\mu)^{4})^{\frac{3}{4}} + (EZ_1^4)^{\frac{3}{4}} +  (E Y^4)^{\frac{3}{4}} \right).
$$



\noindent In the same way we have 
$$
E(Z_{1}^{4}) \leq \  E\left({(X_{2^{n}}-2^n\mu)}^{4}\right)  \left( 1 + \frac{K}{2^{n/4}} \right)
$$
for some universal finite $K$. Putting all together and setting $V_n = \frac{E[(X_{2^n}-2^n \mu)^4]}{2^{2n}}$
we see that,
$$V_{n+1} \leq \frac{V_n}{2} (1+\frac{K}{2^{n/4}}) +6\frac{\alpha^4}{4} (1+ O(1))+ 
\frac{{V_n}^{3/4}}{2^{n/2+2}} Kn^8e^{M^{\prime}\log^4(n)}+ \frac{Kn^{32}e^{4M^\prime \log^4(n)}}{2^{2n+2}},$$
so that $V_n$ satisfies the simpler recursion 
$$
V_{n+1} \le \frac{V_n}{2} ( 1 + \frac{K^\prime}{2^{n/4}})+ K^\prime, 
$$
for suitable constant $K^\prime$, and we get

\begin{lemma}
\label{fourthmo}
For the process $X$ and $\mu $ as in Corollary \ref{comu}
$$
\sup_{n} \frac{E{(X_{2^{n}}-2^{n} \mu)}^{4}}{2^{2n}} < \infty.
$$
\end{lemma}

We now wish to prove that $\alpha $ is strictly positive. 



\begin{proposition} \label{positif}
The constant $\alpha $ defined above is strictly positive.
\end{proposition}

\begin{proof}
In the proof of Proposition \ref{var} (see \eqref{eq-var})  we showed that
$$
2^{n+1} E \left( ( \frac{X_{2^{n+1}}}{2^{n+1}} - \mu)^2  \right) =
2^n E \left( ( \frac{X_{2^{n}}}{2^{n}} - \mu)^2  \right) \ + \ O(n^8e^{M^{\prime}\log^4 (n)}2^{-n/4}) .
$$
Given this, we see at once that it suffices to show that there exists $\beta < 1/4$ so that for each $n_0$ there exists $n_1 \geq n_0 $
so that
$$
E \left( ( X_{2^{n_1}}- 2^{n_1} \mu)^2  \right) \  \geq \ 2^{n_1(1-\beta) }.
$$

To do this we introduce a new regeneration time $\sigma^ \prime $ similar to the regeneration time of order 
$n$ but with two additional stages added into the ``runs". We first choose a $j \in \{-1,1\} $ 
 so that $||g_j ||_\infty \ \ne \ 0$.  Without loss of generality this will be $j=1$.  We define a run beginning at a Markov time $t$ .
If the first five stages are successful then at time $t +n^4 + \log^4(n)   $ the process $\xi $ (relative to $X$) is in approximate equilibrium,
$\nu(n)$ at least close to $X$.
So we have with $b_2(>0)$ probability that $g_1>b_2 $ on the configuration $\xi$ shifted by $X$.

The sixth and seventh stages are motivated by a desire to create a ``regeneration time" $\sigma^ \prime $ so that the distribution of
$\xi $ relative to position $X$ is (essentially) the same irrespective of whether $X$ has advanced by zero or by one during a certain time interval.
This will add uncertainty into the system thus increasing the ``variance".

The primary sixth stage event is that on time interval $[t+n^4+\log^4(n) ,\\t+n^4+\log^4(n)+  1 ]$, we have that either $N^X$ is constant or
increases by one, and the uniform random variable associated to the single Poisson point is in $[1-b_2/M',1]$.

Thus on this event during time interval $[t+n^4+\log^4(n) , t+n^4+\log^4(n)+ 1]$, $X$ either advances by one or stays fixed.
Our task is to show that the process will forget which.

To this end, we also require that on this time interval there is no point in $N^{x,y} $ where one of $\{x,y\}$ is in
$[X(t+n^4+\log^4(n) )-r_0,X(t+n^4+\log^4(n))+r_0] $ and the other is outside.
We also require that $N^X$ is constant on the time interval $[t+n^4+\log^4(n)+ 1, t+n^4+\log^4(n)+  1 + \log^4(n)]$, and that at time
$t+n^4+\log^4(n) + 1$, the process $\xi $ has no $\log^{5/4}(n) $ gaps on the interval
$[X(\tilde t(t,n))-n^9, X(\tilde t(t,n))+n^9]$, where we write (for shortness) $\tilde t(t,n)\stackrel{def}{=}t+n^4+2\log^4(n)+1 $.

Then we define a configuration $\gamma^\prime_{\tilde t(t,n)}$ as with our definition of $\sigma $:\\
\indent

For $|x-X(t+n^4+\log^4(n)  + 1)| \leq n^5 $ we choose $C_x$ to be the condition that at time $\log^4(n)/2 $, the dual
$\hat {\xi}^ {x,\tilde t (t,n)} $
has $h_1 \log^4(n)/2 $ occupied sites in the spatial interval $[ X(t+n^4+\log^4(n) + 1)- n^9, X(t+n^4+\log^4(n)  + 1)+ n^9 ] $.\\

We require that for no $x$ in the above interval do we have $\hat {\xi}^ {x,\tilde t (t,n)} $ survives for time $\log^4(n)/2 $ but
$\gamma^\prime_{\tilde t(t,n)}(x)=0$.


\vspace{0.2cm}


Finally, for the seventh stage  we simply introduce (just as in stage 4 for $\sigma $) an auxiliary uniform random variable $U$. 
We can show via simple arguments that $ \gamma^\prime_{\tilde t(t,n)} =  \xi_{\tilde t(t,n)} $
on $[X(\tilde t(t,n))-n^9, X( \tilde t(t,n))+n^9]$ with probability
$q \ = \ q(\gamma^\prime_{\tilde t (t,n)}, \xi_{ \tilde t (t,n)})$
which will be at least $\frac34$.

The last stage (and hence the ``run") will be a success if this occurs and if $U \leq \frac{3}{4q}$.

Then
relative to $( X(t+n^4+\log^4(n)+    1)) = X(t+n^4+2\log^4(n)) + 1) $ (and independently of
$X(t+n^4+\log^4(n) + 1) - X(t+n^4+\log^4(n)  )$
we have that $\xi_{X(t+n^4+2\log^4(n) +1)}$ has distribution $\nu=\nu(n)$.

As before we produce mostly failures but will with high probability produce a success before time $2^{n/9} $.
We then let the process restart the series of runs and continue. It is the easy to see that for $n$ large
$$
E \left( ( X_{2^{n}}- 2^{n} \mu)^2  \right) \ \geq \ 2^{n7/8 }.
$$
By the first paragraph this concludes the proof.
\end{proof}


We finish this section with a technical result
\begin{lemma}  \label{l2bound}
There exists a constant $c_{13}$ so that for all $n$,
$$
\sup_{t \leq 2^n} E\left( (X_t - t\mu )^2 \right) \ \leq \ c_{13}2^n.
$$
\end{lemma}
\begin{proof}
We need only consider $t \in (2^{n-1}, 2^n)$.  If $t \in (2^{n-1}, 2^{n-1}+ 2.2^{n/4})$, it is easy to see that
$E \left( ( X_{t}- t \mu)^2  \right) \leq K2^n $ for universal $K$ so we need only treat $t \in  (2^{n-1}+ 2.2^{n/4},2^n)$.  In this case we can find
$n-1 = n_1 > n_2 > . . . > n_r $ so that $n_r \geq n/4 -1 $ and
$$
t - \sum_{k=1} ^ r 2^{n_k} \ \in \ (2^{n/4}/4 , 22^{n/4} ).
$$
Given these $n_k $ we construct regeneration times $\sigma _ k $ and processes $(\xi^k , X^k)$ in the manner used in the proof
of Lemma \ref{basic} so that \\
\indent
(i)\,process $(\xi^0, X^0)$ is our given dynamic random walk $(\xi, X)$,\\
\indent
(ii)\,for $k \geq 1$, $\sigma_{k} $ is the $n_{k}$ order regeneration for process $(\xi^{k-1}, X^{k-1})$ after $\sigma_{k-1} + 2^{n_{k-1}}$ (taken to be $0$ if $k=1$).

\noindent The resulting process $((\xi^{k-1})^\prime_s, (X^{k-1})^\prime_s)_{s \geq \sigma_k}$ is written $(\xi^{k}, X^{k})$.

We introduce the following notation: \\
$V_i \ = \ X^{i-1}( \sigma_i ) - X^{i-1}(\sigma_{i-1} + 2^{n_{i-1}}) - 
(\sigma_i - \sigma_{i-1}- 2^{n_{i-1}})\mu,\quad i = 1,2, \dots, r$ and \\
$Y_i \ = \ X^i(\sigma_{i} + 2^{n_{i}})  - X^i( \sigma_i )  - (2^{n_i} )\mu,\quad i = 1,2, \dots, r$\\
(where we adopt the usual convention that $\sigma_0 + 2^{n_0} = 0$) and  \\
$Z \ = \ X_t - t \mu - \sum_{k=1} ^ r (V_i + Y_i )$.

It is not necessary in the definition of $Z$ to assume that $\sigma_{n_r}+ 2^{n_r} $ is less than $t$, 
though the probability that it is not will be less than $e^{-c_6n^3/4^3}$ for $n$ large 
(see Proposition \ref{gencontrol}).  Further let $W \ = \ \sum_{k=1} ^ r V_i  .$
Then we have that
\begin{eqnarray*}
E\left( ( X_t - t \mu)^2 \right) = \ E\left( ( \sum_{k=1}^r Y_i +W+Z)^2 \right)\leq
3 E\left( (\sum_{k=1}^r Y_i)^2 \right)  +  3E(W^2)   +  3E(Z^2).
\end{eqnarray*}
It is easily seen that for some $K$ universal $ E(W^2) $ and $E(Z^2) $ are both bounded by $K2^{n/2}$.
For the other part of the bound, recall that by
Proposition \ref{var} we have that for $n$ large  (and hence $n/4-1 $ large), for each $i$
$$
E(Y_i^2) \leq 2 2^{n_i}\alpha ^ 2.
$$
It then follows by Minkowski inequality
$$
 E\left( (\sum_{k=1}^r Y_i)^2 \right) \ \le  2\alpha^2\left( \sum_{i=1} ^{n-1} 2^{i/2}\right)^{2} \ 
 \leq K 2^{n-1}.
$$


\end{proof}

\begin{corollary}
\label{co-l2bound}
The statement in Lemma \ref{l2bound} applies as well for the dynamic random walk $(\xi,X)$ 
assumed to start with $\xi$ distributed as $\nu(n)$ and $X_0=0$. 
\end{corollary}
\begin{proof}
Indeed it remains to notice that the same proof works, while the only difference regards 
the process $(\xi^{0},X^{0})$ in the first step of the above proof.  
\end{proof}

\begin{proposition}
\label{during}
\noindent Consider a process ${(X_{\sigma +t} -X_{\sigma} - t \mu)}_{t \leq 2^{n}}$ where $\sigma$ is an $n$
order regeneration time. For each $\gamma > 0$, there exists $c_{\gamma}>0$ so that for all $n$ large

\begin{eqnarray*}
P(\sup_{s \leq 2^{n}}\vert X_{\sigma + s} - X_{\sigma} -s \mu \vert \geq 2^{\frac{n(1 + \gamma)}{2}} )\\
\leq\ 2 P( \vert X_{\sigma +2^{n}} - X_{\sigma} -2^{n} \mu \vert
\geq \frac{1}{2} 2^{\frac{n}{2}(1 + \gamma)} ) \ \leq \ \frac{1}{c_\gamma} 2^{-n \gamma}.
\end{eqnarray*}
\end{proposition}

\begin{proof}
Let $T=\inf\{s >0\colon \vert X_{\sigma + s} - X_{\sigma} -s \mu \vert \geq 2^{\frac{n(1 + \gamma)}{2}}\} 
\wedge 2^n$. 
This is a stopping time for the Harris system filtration.  If $2^n -T < n^9e^{M\log^4(n)}$, then there is 
hardly anything to prove and so we suppose otherwise.
At time $T$ we begin runs concluding in an $n$ order regeneration $\sigma $.  We
put $Z = X^\prime_{2^n}- X^\prime_{\sigma}- (2^n -\sigma ) \mu$ and define random variable $W$ by
$$
X_{2^n} - 2^n \mu \ = \ (X_T -T \mu ) \ + \ ((X_\sigma - X_T) - (\sigma -T) \mu) \ + \  Z \ + \ W,
$$
so
$$
| X_{2^n} - 2^n \mu | \ \geq  \ | X_T -T \mu |  \ - \  | (X_\sigma - X_T) - (\sigma -T) \mu |  \ - \  | Z |  \ - \ | W| ,
$$
$$
 \geq  \ 2^{\frac{n(1 + \gamma)}{2}}  \ - \  | (X_\sigma - X_T) - (\sigma -T) \mu |  \ - \  | Z |  \ - \ | W|.
$$
By elementary bounds on regeneration times and Poisson process tail probabilities
and Proposition \ref{gencontrol}, we have that outside probability $2e^{-c_6n^3}$ for $n $ large
$$
|(X_\sigma - X_T) - (\sigma -T) \mu | \ \leq \ Mn^9e^{M^{\prime}\log^4(n)} \ << 2^{\frac{n(1 + \gamma)}{2}}.
$$

Secondly,  given information up to $ T $, the term $Z$ is equal in distribution to $X_s-s \mu $ 
for $s = 2^n - \sigma $ where the $(\xi,X)$ process begins with $\xi$ in distribution $\nu(n)$ at least when 
restricted to the sites within $n^9 $ of $X_0=0$. We may then apply Corollary \ref{co-l2bound} to see that 
for suitable universal constant $K$  
\begin{equation}
\label{14out14}
P\left( |Z| \geq 2^{\frac{n(1 + \gamma)}{2}} / 4 \right)\ \leq \ K2^{-n \gamma}.
\end{equation}
Thus we obtain (at least for large $n$), using the usual bounds as before for the probability that the random variable $W$ is zero,
\begin{eqnarray*}
P(\sup_{s \leq 2^{n}}\vert X_{\sigma + s} - X_{\sigma} -s \mu \vert \geq 2^{\frac{n(1 + \gamma)}{2}}  )\\
\le 2P (|X_{2^n} - 2^n \mu | \ \geq \ 2^{\frac{n(1 + \gamma)}{2}} / 4),
\end{eqnarray*}
and we are done.
\end{proof}

\section{Proof of Theorem \ref{thm1}}

\noindent Given this we can establish our invariance principle. We consider $2^{n} \leq t< 2^{n+1}$ 
and choose scale $2^{\frac{n}{2} (1 + \beta ) }  \ = \ 2^{n_{1}}$ for $0 < \beta < 1$.
We can apply Proposition \ref{gencontrol} to show that if we define $n_1$ order regeneration times 
$\sigma_k $ recursively, as with the proof of Lemma \ref{basic}, so that
$\sigma_k $ is the time of the first regeneration for process $(\xi^{k-1}, X^{k-1})$ after starting 
runs at time $\sigma_{k-1} + 2^{n_1} $, and
$$
 (\xi^{k}_{\sigma_k}, X^{k}_{\sigma_k})   \ = \ (\xi^{k-1}, X^{k-1})^\prime
$$
then we have with high probability that for all $\sigma_k < t $ that
$$
X_{\sigma_k + 2^{n_1} } - X_{\sigma_k} \ = \ X ^ {k}_{\sigma_k + 2^{n_1} } - X^k_{\sigma_k }.
$$

We decompose the motion ${(X_{s})}_{s \leq t}$ into its increments over an alternating series of intervals
$I_{1}, J_{1}, I_{2}, J_{2},\cdots$, where $I_{k} = [\sigma_{k-{1}}+2^{n_{1}}, \sigma_{k}]$ for $n_{1}$ 
regeneration times $\sigma_k$, and $J_{k} = [\sigma_{k}, \sigma_{k}+2^{n_{1}}]$.

\noindent We have that $\sigma_{k} \geq t$ for $k=k_{0} = \left[ \frac{t}{2^{n_{1}}} \right]$ and 
(outside very small probability) $\sigma_{k} < t$ for $k = k_{1} = \left[ \frac{t}
{ 2^{n_{1}}+  n^8e^{M\log^4(n)}} \right]$.

\noindent Thus via the usual invariance principle and Berry Esseen bounds (see e.g. \cite{Du}) we have 

(A)$\sum_{k=1}^{k_{1}} \frac{(X_{\sigma_{k} + 2^{n_{1}}}- X_{\sigma_{k}}- \mu 2^{n_{1}}    )}{\sqrt{t}}
\stackrel{D}{\to} N(0, \alpha^{2})$,


(B)$\sum_{k=1}^{k_{0}} \frac{(X_{\sigma_{k} + 2^{n_{1}}}- X_{\sigma_{k}}- \mu 2^{n_{1}}    )}{\sqrt{t}}
\stackrel{D}{\to} N(0, \alpha^{2})$,

and

(C) $W_t\stackrel{def}{=}\sup_{k_{1} \leq k' \leq k_{0}} \left \vert  \sum^{k'}_{k=k_{1}} \frac{(X_{\sigma_{k}
+ 2^{n_{1}}}- X_{\sigma_{k}} -2^{n_{1}} \mu)}{\sqrt{t}}   \right \vert \stackrel{p{r}}{\to} 0$.

\noindent Furthermore we have that with probability tending to $1$ at $n \to \infty$
(with the usual convention for $\sigma_{0} +2^{n_{1}}$):

$$W^{1}_{t}\stackrel{def}{=}\sum_{k=1}^{k_{0}} \left \vert \frac{(X_{\sigma_{k}}-
X_{\sigma_{k-1}+2^{k}} -(\sigma_{k} - \sigma_{k-1} + 2^{n_{1}}) \mu}{\sqrt{t}}   \right \vert
\leq K  n^8e^{M^\prime\log^4(n)} \frac{2^{n(1-\frac{\beta}{2})}}{2^{\frac{n}{2}}},
$$
which then tends to zero as $n \to \infty$.

\noindent Furthermore by Proposition \ref{during} we have for
$k_{2} = \sup \{k: \sigma_{k} < t \}$, that with probability that tends to one as $n \to \infty$ 

$$W^{2}_{t}\stackrel{def}{=}\sup_{\sigma_{k_{2}} \leq s \leq \sigma_{k_{2}}+ 2^{n_{1}}}
\left \vert \frac{(X_s - X_{\sigma_{k_{2}}} -(s- \sigma_{k_{2}} ) \mu}{\sqrt{t}} \right \vert
\leq \frac{2^{ \frac{n(1+ \beta)(1+ \gamma)}{4}    }}{2^{\frac{n}{2}}} < 2^{-n \epsilon }
$$
for $\frac{(1+ \beta)(1+ \gamma)}{2} < 1$ and $ \epsilon = 1- \frac{(1+ \beta)(1+ \gamma)}{2}$.

\noindent Then we have
$$\left \vert \frac{X_{t} - t \mu}{\sqrt{t}}- \sum^{k_1}_{k=1}
\frac{(X_{\sigma_{k} + 2^{n_{1}}}- X_{\sigma_{k}} -2^{n_{1}} \mu )}{\sqrt{t}} \right \vert
\leq W_{t} + W_{t}^{1} + W^{2}_{t},
$$
and the desired convergence follows.\qed

 \vspace {1 cm}

\noindent Acknowledgement: Research partially supported by CNPq 402215/2012-5. \\M.E.V. partially supported by CNPq 304217/2011-5.



\begin{thebibliography}{99}







\bibitem[Du1]{Du1} R.\ Durrett. Oriented percolation in two dimensions. \emph{Ann. Probab.}  \textbf{12},
999-1040 (1984).

\bibitem[Du]{Du} R.\ Durrett, \textit{Probability: theory and examples} (4th ed.). Cambridge University Press (2010).

\bibitem[DG]{DG} R.\ Durrett, D. Griffeath.  Supercritical contact processes on Z. \emph {Ann. Probab.} \textbf{11},
1-15 (1983).

\bibitem[dHS]{dHS} F.\ den Hollander, R.\ dos Santos. Scaling of a random walk on a
supercritical contact process. Preprint.

\bibitem[dHSS]{dHSS} F.\ den Hollander, R.\ dos Santos, V.\ Sidoravicius. Law of large numbers for non-elliptic random walks in dynamic random environments.


\bibitem[DSch]{DSch} R.\ Durrett, R. H.\ Schonmann. Large deviations for the contact process and two-dimensional percolation. \emph{Probab. Theory Rel. Fields}
 \textbf {77}, no. 4, 583-603 (1988).


\bibitem[GP]{GP} A.\ Galves, E.\ Presutti. Edge fluctuations for the one-dimensional supercritical contact process. Ann. Probab. 15 (1987), no. 3, 1131–1145.

\bibitem[K]{K}  T.\ Kuczek.  The Central Limit Theorem for the Right Edge of Supercritical Oriented Percolation.
Ann. Probab. \textbf{17}, 1322-1332  (1989).

\bibitem[Li1]{lig85} T. Liggett, \emph{Interacting particle systems}. Grundlehren der mathematischen Wissenschaften \textbf{276}, Springer (1985).











\end{thebibliography}
\end{document}